\documentclass[12pt,a4paper]{amsart}
\usepackage[latin1]{inputenc}
\usepackage{amssymb, amsmath}
\usepackage{geometry}
\usepackage{enumerate}
\usepackage[colorlinks=true,linkcolor=blue,urlcolor=blue,citecolor=blue]{hyperref}
\theoremstyle{plain}
\newtheorem{theorem}{Theorem}[section]
\newtheorem{corollary}[theorem]{Corollary}

\newtheorem{proposition}[theorem]{Proposition}

\theoremstyle{definition}
\newtheorem{definition}[theorem]{Definition}

\newtheorem	{question}{Question}

\theoremstyle{remark}
\newtheorem*{example}{Example}

\newcommand{\Ss}{\mathcal{S}}
\newcommand{\K}{\mathcal{K}}
\newcommand{\LL}{\mathcal{L}}
\newcommand{\abs}[1]{\lvert#1\rvert}
\newcommand{\norm}[1]{\lVert#1\rVert}

\newcommand{\bignorm}[1]{\bigl\lVert#1\bigr\rVert}
\newcommand{\Bigabs}[1]{\Bigl\lvert#1\Bigr\rvert}
\newcommand{\Bignorm}[1]{\Bigl\lVert#1\Bigr\rVert}
\renewcommand{\le}{\leqslant}

\newcommand{\term}[1]{{\textit{\textbf{#1}}}}

\makeatletter
\def\@tvsp{\mathchoice{{}\mkern-4.5mu}{{}\mkern-4.5mu}{{}\mkern-2.5mu}{}}
\def\ltrivert{\left|\@tvsp\left|\@tvsp\left|}
\def\rtrivert{\right|\@tvsp\right|\@tvsp\right|}
\def\bltrivert{\bigl|\@tvsp\bigl|\@tvsp\bigl|}
\def\brtrivert{\bigr|\@tvsp\bigr|\@tvsp\bigr|}
\makeatother

\DeclareMathOperator{\Range}{Range}
\DeclareMathOperator{\dist}{dist}

\renewcommand{\mid}{\::\:}

\begin{document}

\title[Disjointly homogeneous Banach lattices]{Disjointly homogeneous Banach lattices and applications}

\author[Julio Flores]{Julio Flores}
\address{Departamento de Matem\'{a}tica Aplicada, Escet, Universidad Rey
          Juan Carlos\\ 28933\\ M\'ostoles, Madrid, Spain.}
\email{julio.flores@urjc.es}

\author[Francisco. L. Hern\'{a}ndez]{Francisco L. Hern\'{a}ndez}
\address{Departamento de An\'{a}lisis Matem\'{a}tico, Universidad Complutense de Madrid\\ 28040\\ Madrid, Spain.}
\email{pacoh@mat.ucm.es}

\author[Pedro Tradacete]{Pedro Tradacete}
\address{Department of Mathematics\\ Universidad Carlos III de Madrid\\ 28911\\ Legan\'es, Madrid, Spain.}
\email{ptradace@math.uc3m.es}

\thanks{The research has been supported by the Spanish Ministerio de Econom\'{\i}a y Competitividad through grant MTM2012-31286 and Grupo UCM 910346. P. Tradacete has also been partially supported by MTM2010-14946.}



\begin{abstract}
This is a survey on disjointly homogeneous Banach lattices and their applicactions. Several structural properties of this class are analyzed. In addition  we show how these spaces provide a natural framework for studying the compactness of powers of operators allowing for a unified treatment of well-known results.
\end{abstract}

\maketitle

\tableofcontents

\section{Introduction}
This paper is a survey on the properties  of the recently introduced class of  disjointly homogeneous Banach lattices, as well as the compactness properties of the operators defined on these spaces. It collects information previously from the papers (\cite{FTT2, FHST, FHSTT, HST}) and some new facts.

Recall that a Banach lattice $E$ is \em disjointly homogeneous \em if two arbitrary sequences of normalized pairwise disjoint elements in $E$ always have equivalent subsequences. The motivation  which gave rise to the definition of a disjointly homogeneous space was to decide the  compactness  of the iterations of a given operator which already enjoyed nice close-to-compactness properties. Thus the first and third authors together with V. G. Troitsky considered this notion for the first time in \cite{FTT2}.

Later on E. M. Semenov and the authors (\cite{FHST}) analyzed the general problem of obtaining compactness of the iterations of a strictly singular operator on a Banach lattice, extending the classical result by V. D. Milman
(\cite{Milman}) which states that strictly singular operators in $L_p(\mu)$, $1\leq p\leq\infty,$ have compact square. In fact, one of the purposes of this survey is to offer compelling evidence  that the class of disjontly homogeneous Banach lattices constitutes a proper setting for treating these questions. It is particularly evident in connection with the Kato property, i.e. when the class of compact and strictly singular operators coincide (\cite{HST}). From the point of view of the structural properties of the class of disjointly homogeneous Banach lattices, several aspects have been explored by  V. G. Troitsky and E. Spinu jointly with the authors in \cite{FHSTT}; particularly two of these aspects have been  studied in detail, namely  the problem of self-duality and the problem of obtaining complemented copies of the span of disjoint sequences.

The paper is organized in two clearly differentiated parts. The first one includes sections one through five which focus on disjointly homogeneous Banach lattices themelves. Definitions and examples are given, and structure properties such as the self-duality of this class and the existence of complemented copies of disjoint sequences are addressed.   The second part, including the remaining sections, focuses on the operators defined on disjointly homogeneous Banach lattices and the properties they have; particularly, attention is given  to the compactness properties of the iterations of endomorphisms    as well as the relation between disjointly homogeneous Banach lattices and the Kato property. The paper concludes with a list of some open questions.

We follow standard terminology concerning Banach spaces and Banach lattices as in the monographs \cite{AB, LT1, LT2, MN}. In the sequel by an operator we  always mean a bounded linear operator. Given a sequence $(x_n)$ in a Banach space, we write $[x_n]$ for the closed linear span of the sequence. Given basic sequences $(x_n)$, $(y_n)$, and $C>0$, the notation $(x_n)\overset{C}{\sim}(y_n)$ means that for every scalars $(a_n)_{n=1}^\infty$
$$
  C^{-1}\Bignorm{\sum_{n=1}^\infty a_n y_n}
  \le \Bignorm{\sum_{n=1}^\infty a_n x_n}
  \le C\Bignorm{\sum_{n=1}^\infty a_n y_n}.
$$

\section{Disjointly homogeneous Banach lattices: definition and examples}\label{definitions and examples}

The notion of disjointly homogeneous Banach lattice was first introduced in \cite{FTT2}; let us recall its definition.

\begin{definition}
\label{DH definition}
A Banach lattice $E$ is \term{disjointly homogeneous (DH)} if for every pair $(x_n)$, $(y_n)$ of normalized disjoint sequences in $E$, there exist $C>0$ and a  subsequence $(n_k)$ such that $(x_{n_k})\overset{C}{\sim}(y_{n_k})$.
\end{definition}

Our interest will focus on those Banach lattices for which there is $1 \le  p\le \infty$ such that every normalized disjoint sequence $(x_n)$ has a subsequence $(x_{n_k})$ equivalent to the unit vector basis of $\ell_p$ (or $c_0$ for $p=\infty$), i.e,

$$
  C^{-1}\big(\sum_{k=1}^\infty |a_k|^p\big)^{1/p}\le
\Bignorm{\sum_{k=1}^\infty a_k x_{n_k}}
  \le C(\sum_{k=1}^\infty |a_k|^p)^{1/p},
$$
for some $C>0$.
 These form an important class of DH spaces, which will be denoted \term{$p$-disjointly homogeneous}, in short $p$-DH (resp. $\infty$-disjointly homogeneous, in short $\infty$-DH). Clearly, $L_p$-spaces are $p$-DH.

Note that 1-DH Banach lattices have been considered previously under a different approach. Recall that a Banach lattice $E$ has the \emph{positive Schur property} if every weakly null sequence $(x_{n})$ of positive vectors is norm convergent, see \cite{Kaminska-Mastylo,wnuk_glasgow,wnuk_RMC,wnuk_survey}. It follows from, e.g.~\cite[Corollary~2.3.5]{MN}, that it suffices to verify this condition for disjoint sequences. Using Rosenthal's $\ell_1$-theorem, it was proved in \cite{FHSTT} that a Banach lattice $E$ is 1-DH if and only if $E$ has the positive Schur property.

Observe that, in the definition of a DH Banach lattice, it is important to allow for the possibility of passing to subsequences in order to get the required equivalence. Otherwise, the class reduces to the spaces $L_p(\mu)$ or $c_0(\Gamma)$ (\cite[Proposition 2.2]{FHSTT}).

Thus $L_p(\mu)$-spaces exhibit a particularly strong version of this definition. But these are not the only examples. For instance, in the context of function spaces, Lorentz  spaces $\Lambda(W,q)$ and $L_{p,q}$ on $[0,1]$ are
$q$-DH.

Recall that given $1\leq q<\infty$ and $W$ a positive, non-increasing function in $[0,1]$, such that $\lim_{t\rightarrow0}W(t)=\infty,\, W(1)>0$ and
$\int_0^1W(t)dt=1$, the \textit{Lorentz function space} $\Lambda(W,q)[0,1]$ is the space of all measurable functions $f$ on $[0,1]$ such that
$$\|f\|=\Big(\int_0^1 f^*(t)^q W(t) dt\Big)^{1/q}<\infty,$$ where $f^*$ denotes the decreasing rearrangement of the function $f$ (cf \cite[Chapter 2]{LT2}).
Let us also recall that for $1< p< \infty$ and $1\leq q\leq\infty$, the Lorentz space $L_{p,q}[0,1]$ is the space of all measurable functions $f$ in $[0,1]$ such that
$$\|f\|_{p,q}=\left\{
                \begin{array}{ll}
                  \bigg(\int_0^{\infty}(t^{1/p}f^*(t))^q\frac{dt}{t}\bigg)^{1/q}<\infty &\textrm{ for }1\leq q<\infty, \\
                  \underset{t>0}{\sup}\,\,t^{1/p}f^*(t)<\infty, &\textrm{ if }q=\infty. \\
                \end{array}
              \right.$$

The  following (see \cite{Carothers-Dilworth}, \cite[Proposition 5.1]{FJT} ) shows that these belong to the class of DH spaces.

\begin{proposition}\label{CarothersDilworth}
Let $1\leq q<\infty$. Let $(f_n)_n$ be a disjoint normalized sequence in $\Lambda(W,q)[0,1]$	
(resp. $L_{p,q}[0,1]$). For each $\varepsilon>0$, there exists a
subsequence $(f_{n_k})$ which is $(1+\varepsilon)$-equivalent to
the unit vector basis of $\ell_q$, whose span is a complemented
subspace of $\Lambda(W,q)[0,1]$ (resp. $L_{p,q}[0,1]$).
\end{proposition}

For the maximal Lorentz spaces $L_{p,\infty}[0,1]$, $1<p<\infty$, the situation is different. Indeed, the  space $L_{p,\infty}[0,1]$ satisfies that every disjoint  sequence in its order continuous part $(L_{p,\infty}(0,1))^o$ (the closed linear span of the characteristic functions in $L_{p,\infty}[0,1]$) has a subsequence equivalent to the unit vector basis of $c_0$ (see \cite{Novikov}). But there exists a disjoint normalized sequence $(f_n)$ in $L_{p,\infty}[0,1]$ equivalent to the unit basis of $\ell_p$ (which generates a complemented subspace) (see \cite{FHST}). Therefore, $L_{p,\infty}[0,1]$ is not DH.

In the class of  Orlicz spaces we have further examples of DH spaces. Recall that given an Orlicz function $\varphi:\mathbb{R}_+\rightarrow \mathbb{R}_+$, the \textit{Orlicz function space} $L_\varphi(\Omega,\Sigma,\mu)$ is the space of all $\Sigma$-measurable functions $f$ on $\Omega$ such that $\int_\Omega\varphi\Big(\frac{|f|}{r}\Big)d\mu <\infty$ for some $r>0$. This is a Banach lattice endowed with the  Luxemburg norm
$$\|f\|_{L_\varphi}=\inf\bigg\{r>0:\int_\Omega\varphi\Big(\frac{|f|}{r}\Big)d\mu \leq 1\bigg\}.$$

A characterization of DH Orlicz spaces, over finite \cite{FHST} and infinite \cite{FHSTT} measure spaces, is known. In order to state this, let us first recall the definition of certain subsets of the space of continuous functions $C[0,1]$ associated to the Orlicz function $\varphi$ (see \cite{LT3}):

$$E_{\varphi,s}^\infty=\overline{\Big\{\frac{\varphi(r\cdot)}{\varphi(r)}:r\geq
s\Big\}},\,\,\,\,\,\,\,\,E_\varphi^\infty=\bigcap_{s>1}
E_{\varphi,s}^\infty\,,\, \textrm{ and
}C_\varphi^\infty=\overline{conv}(E_\varphi^\infty).$$

Similarly,  let
\begin{displaymath}
  E_{\varphi}(0,\infty)=\overline{\Bigl\{
    F \in C[0,1]\mid F(\cdot)=\frac{\varphi(s\cdot)}{\varphi(s)}\, , \textrm{ for some }\, s\in (0,\infty)
  \Bigr\}},
\end{displaymath}
and $C_{\varphi}(0,\infty)=\overline{\rm conv}\, E_{\varphi}(0,\infty)$, in the space $C[0,1]$.

As usual, for a subset $A\subset C[0,1]$ and a function $h$, we will write $A\cong\{h\}$ whenever every function in $A$ is equivalent to the function $h$ at $0$.

\begin{theorem}\label{DH-Orlicz}
$ $
\begin{enumerate}
\item A separable Orlicz space $L_\varphi[0,1]$ is DH if and only if $E_{\varphi}^\infty\cong\{t^p\}$ (for some $1\le p<\infty$).
\item A separable Orlicz space $L_{\varphi}(0,\infty)$ is DH if and only if $C_{\varphi}(0,\infty)\cong\{t^{p}\}$ (for some $1\le p<\infty$).
\end{enumerate}
Moreover, in each case the space is $p$-DH for the corresponding $p$.
\end{theorem}

The proof of the finite measure case is based on techniques from \cite{LT3}. For instance, if $\varphi(x)=x^plog(1+x)$, for $1\le p<\infty$, then  $L_\varphi[0,1]$ is $p$-DH.

For the infinite measure case, among other things, the proof makes use of \cite[Theorem 1.1]{Nielsen}, which asserts that if an Orlicz function $F$ is equivalent to a function in $C_{\varphi}(0,\infty)$ then $L_{\varphi}(0,\infty)$ contains a lattice copy of the Orlicz sequence space $\ell_F$ and, conversely, every normalized disjoint sequence in  $L_{\varphi}(0,\infty)$ contains a subsequence equivalent to the unit vector basis of $\ell_F$ for some $F\in C_{\varphi}(0,\infty)$.

In the discrete setting the class of DH Banach lattices is considerable smaller. Clearly, it contains the spaces $c_0$ and $l_p$, $1\le p\le\infty$; by contrast, Orlicz and Lorentz sequence spaces other than $\ell_p$  cannot be DH. This follows from the well known fact that $E$ is a stable space (\cite{Krivine-Maurey}). Indeed, if one starts with  a given pairwise disjoint sequence $(x_n)$ in $E$, then the stability implies that there 	is some block sequence $(w_n)$ of $(x_n)$  equivalent to the unit vector basis of some $\ell_p$ for $1\le p<\infty$. If $E$ is assumed to be DH, then for some subsequence $(n_k)$, $(x_{n_k})$ and $(w_{n_k})$ are equivalent to the unit vector basis of $\ell_p$. But the unit basis in $E$ is symmetric; thus, it must be  equivalent to the unit vector basis of $\ell_p$.

Tsirelson space also falls within the category of DH Banach lattices, as shown in \cite{FTT2}. As a  consequence, we deduce that DH Banach lattices need not be $p$-DH for any $1\le p\le \infty$. Some modifications of Tsirelson space, together with Baernstein and Schreier spaces (cf. \cite{Casazza-Shura}) are easily seen to be also DH.

Observe that in the definition of a DH Banach lattice it is enough to consider only positive disjoint normalized (or even semi-normalized) sequences. A formally weaker version of DH has been considered also in \cite{FHSTT}: namely, a Banach lattice is \emph{quasi-DH} if any two sequences of disjoint elements $(x_n)$ and $(y_n)$ have equivalent subsequences. This means that $(x_{n_k})\sim(y_{m_k})$ for some, non necessarily equal, subsequences $(n_k)$ and  $(m_k)$. The following result follows from a standard application, based on \cite{Schlumprecht}, of the infinite Ramsey theorem, and solves a natural question posed in \cite{FHSTT}.

\begin{proposition}
A Banach lattice is DH if and only if it is quasi-DH.
\end{proposition}

\begin{proof}
We prove that a quasi-DH Banach lattice $X$ is DH. Let $(x_n)$ be a disjoint sequence in $X$. For an infinite set $A$, by $\mathcal{P}_\infty(A)$ we denote the family of infinite subsets of $A$. We claim that  $\mathcal{P}_\infty(\mathbb{N})$ contains some set $\mathbb{M}=\{m_k:k\in\mathbb{N}\}$ with $m_1<m_2<\ldots$  such that for every infinite subset $\mathbb{P}=\{p_j:j\in\mathbb{N}\}\subset\mathbb{M}$, the equivalence $(x_{p_{2j}})\sim(x_{p_{2j+1}})$ holds.

Indeed, let
$$
\mathcal{S}=\Big\{\{m_k:k\in\mathbb{N}\}\in\mathcal{P}_\infty(\mathbb{N}): \,\forall k,\,m_k<m_{k+1},\,\textrm{and}\,(x_{m_{2k}})\sim(x_{m_{2k+1}})\Big\}.
$$
It is easy to check that $\mathcal{S}$ is a Borel subset of $\mathcal{P}_\infty(\mathbb{N})$. By the Galvin-Prikry Theorem (cf. \cite{Diestel}), there is $\mathbb{M}\in\mathcal{P}_\infty(\mathbb{N})$ such that either $\mathcal{P}_\infty(\mathbb{M})\subset\mathcal{S}$ or $\mathcal{P}_\infty(\mathbb{M})\cap\mathcal{S}=\emptyset$. Now, suppose that $\mathcal{P}_\infty(\mathbb{M})\cap\mathcal{S}=\emptyset$. Since $X$ is quasi-DH the disjoint sequences $(x_{m_{2k}})$ and $(x_{m_{2k+1}})$ have equivalent subsequences, that is $(j_k)$, $(l_k)$ such that
$$
(x_{m_{2j_k}})\sim(x_{m_{2l_k+1}}).
$$
Passing to further subsequences we have that either $2j_1<2l_1+1<2j_2<2l_2+1<\ldots$ or $2l_1+1<2j_1<2l_2+1<2j_2<\ldots$. In both cases we have that
$$
I=\{m_{2j_1},m_{2l_1+1},m_{2j_2},m_{2l_2+1},\ldots\}\in\mathcal{P}_\infty(\mathbb{M})\cap\mathcal{S}.
$$
This contradiction implies that $\mathcal{P}_\infty(\mathbb{M})\subset\mathcal{S}$, and the claim follows.

To finish the proof, let $(x_n)$ and $(y_n)$ be two sequences of normalized disjoint elements in $X$. By the claim, we can assume, passing to some subsequence, that both $(x_n)$ and $(y_n)$ run on $\mathbb{M}$ and also that  $(x_{m_{2k}})\sim(x_{m_{2k+1}})$ for every $(m_k)$ with $m_1<m_2<\ldots$ Since $X$ is quasi-DH, there exist $(n_k)$ and $(p_k)$ such that $(x_{n_k})\sim(y_{p_k})$. Passing to a further subsequence, we can assume that $n_1<p_1<n_2<p_2<\ldots$ or $p_1<n_1<p_2<n_2<\ldots$ By the properties of the sequence $(x_n)_{n\in\mathbb{M}}$,  it follows that
$$
(y_{p_k})\sim(x_{n_k})\sim(x_{p_k}).
$$
\end{proof}

\section{Duality for disjointly homogeneous Banach lattices}\label{stab-dual}

It is natural to inquire about the stability by duality of the class of DH Banach lattices. Note that by Proposition \ref{CarothersDilworth}, for $1<p<\infty$, the Lorentz space $L_{p,1}[0,1]$ is 1-DH. However, as mentioned above, its dual $L_{p',\infty}[0,1]$ is not DH  (here $\frac{1}{p}+\frac{1}{p'}=1$). Thus, in the non-reflexive case, the class of DH Banach lattices is not stable under duality.

By contrast, in the reflexive case, all the examples of DH Banach lattices mentioned above  have DH duals. This is obviously true for $L_p(\mu)$ spaces with  $1<p<\infty$ as well as for Lorentz spaces since Proposition \ref{CarothersDilworth}   can also  be applied to their duals. In addition, as shown in Theorem \ref{DH-Orlicz},  an Orlicz space $L_\varphi[0,1]$ is DH if and only if every function in the set $E_\varphi^\infty$ is equivalent to the function $t^p$ for some fixed $1\le p<\infty$. Note however that if  $\varphi'$ denotes the conjugate Orlicz function of $\varphi$, then every function in $E_{\varphi'}^\infty$ is easily seen to be equivalent to the function $t^{p'}$, which again is tantamount to the space $L_\varphi[0,1]^*=L_{\varphi'}[0,1]$ being DH.

Therefore, based on these examples, one might reasonably conjecture that among  reflexive spaces being DH is indeed a self dual property. As it will be shown this turns out to be false. Still it holds true under additional assumptions which are of interest. The rest of the section is devoted to clarifyig this.

A natural approach to proving  a positive result of stability by duality should look more or less like this: start with  two arbitrarily chosen  disjoint normalized sequences $(x_n)$ and $(y_n)$ in a reflexive Banach lattice $E$ whose dual $E^*$ is DH. We would like to prove that, up to passing to some subsequence, $(x_n)$ and $(y_n)$ are equivalent. Certainly,  two disjoint normalized sequences $(x_n^*)$ and $(y_n^*)$ in $E^*$ can be found in $E^*$ such that $x_n^*(x_m)=y_n^*(y_m)=\delta_{nm}$ for each $n,m\in\mathbb N$. Since $E^*$ is DH, after passing to subsequences we may assume that $(x_n^*)$ and $(y_n^*)$ are equivalent in $E^*$. On the other hand, for each $m$, we can consider $x_m^*$ as a functional on $[x_n]$ (formally speaking, we are taking the restriction of $x_m^*$ to $[x_n]$); moreover, since $E$ is reflexive, $(x_m^*)$ is a basis of $[x_n]^*$. Then for any coefficients $\alpha_1,\dots,\alpha_m$ we have
\begin{eqnarray*}
  \bignorm{\sum_{i=1}^m\alpha_ix_i}&=&\sup\Bigl\{\Bigabs{\bigl\langle \sum_{i=1}^m\alpha_ix_i,\sum_{i=1}^m\beta_ix_i^*\bigr\rangle}\mid
     \bignorm{\sum_{i=1}^m\beta_ix_i^*}_{[x_n]^*}\le 1\Bigr\}\\
    &=&\sup\Bigl\{\Bigabs{\sum_{i=1}^m\alpha_i\beta_i}\mid
     \bignorm{\sum_{i=1}^m\beta_ix_i^*}_{[x_n]^*}\le 1\Bigr\}.
\end{eqnarray*}
In general, clearly $\bignorm{\sum_{i=1}^m\beta_ix_i^*}_{[x_n]^*}\le\bignorm{\sum_{i=1}^m\beta_ix_i^*}_{E^*}$. However, if we could somehow control the converse estimate, we could continue, using the equivalence of $(x_n^*)$ and $(y_n^*)$ in $E^*$ as follows
\begin{eqnarray*}
  \bignorm{\sum_{i=1}^m\alpha_ix_i}&  \sim & \sup\Bigl\{\Bigabs{\sum_{i=1}^m\alpha_i\beta_i}\mid
     \bignorm{\sum_{i=1}^m\beta_ix_i^*}_{E^*}\le 1\Bigr\}\\
    &\sim&\sup\Bigl\{\Bigabs{\sum_{i=1}^m\alpha_i\beta_i}\mid
     \bignorm{\sum_{i=1}^m\beta_iy_i^*}_{E^*}\le 1\Bigr\}
    \sim\bignorm{\sum_{i=1}^m\alpha_iy_i},
\end{eqnarray*}
which would imply that $(x_n)$ and $(y_n)$ are equivalent. In particular, such an argument would work if we could find a bounded operator $S\colon[x_n]^*\to E^*$ such that $Sx_m^*=x_m^*$ for each $m$ and a similar operator for $(y_n)$. The previous discourse is collected in the following

\begin{definition}\label{Prop P}
A Banach lattice $E$ has the \term{ $\mathfrak{P} \ property$} if for every disjoint positive normalized sequence $(f_n)\subset E$ there exists an operator $T:E\rightarrow [f_n]$, such that some subsequence $(T^*f_{n_k}^*)$ is equivalent to a seminormalized disjoint sequence in $E^*$ (here $(f_n^*)$ denote the corresponding biorthogonal functionals in $[f_n]^*$).
\end{definition}

Given a disjoint sequence $(f_n)$ as in the above definition, we can consider $Px=\sum_{k=1}^\infty f_{n_k}^*(x)f_{n_k}$, the canonical projection from $[f_n]$ onto $[f_{n_k}]$ (which has $\norm{P}= 1$ because $(f_n)$ is 1-unconditional). If $E$ has the  $\mathfrak{P}$ property,  then we can now view
  \begin{displaymath}
    PTx
    =\sum_{k=1}^\infty f_{n_k}^*(Tx)f_{n_k}
    =\sum_{k=1}^\infty \bigl(T^*f_{n_k}^*\bigr)(x)f_{n_k}
  \end{displaymath}
as a bounded operator on $E$.

The $\mathfrak{P}$ property can be characterized as follows (\cite[Proposition 3.3]{FHSTT}):

\begin{proposition}
Let $E$ be a reflexive Banach lattice. The following are equivalent:
\begin{enumerate}
\item\label{distball} For every disjoint positive normalized sequence $(f_n)\subset E$ there exists a positive operator $T:E\rightarrow [f_n]$, with $\liminf_n\dist\bigl(f_n,T(B_E)\bigr)<1$.
\item\label{nonnull} For every disjoint positive normalized sequence $(f_n)\subset E$ there exists a positive operator $T:E\rightarrow [f_n]$, such that $\norm{T^*f_n^*}\nrightarrow0$.
\item\label{P} $E$ has the  $\mathfrak{P}$ property.
\end{enumerate}
\end{proposition}

Notice that  Banach lattices in which every disjoint positive sequence has a subsequence whose span is complemented by a positive projection  satisfy the $\mathfrak{P}$ property. Examples of these include $L_p$ spaces, Lorentz function spaces $\Lambda(W,p)$, Tsirelson's space, etc.

As intended, the assumption of the $\mathfrak{P}$ property yields a partial positive answer  to the problem of stability by duality of DH Banach lattices.

\begin{theorem}\label{duality lemma P}
Let $E$ be a reflexive Banach lattice with the  $\mathfrak{P}$ property. If $E^*$ is DH, then $E$ is DH. Moreover, in the particular case when $E^*$ is $p$-DH, for some $1<p<\infty$, then $E$ is $q$-DH with $\frac1p+\frac1q=1$.
\end{theorem}

This fact, which was given in \cite{FHSTT}, can be used in particular to show that if a reflexive Banach lattice $E$ is $p$-DH and satisfies a lower $p$-estimate, for some $1<p<\infty$, then $E^*$ is $q$-DH (with $\frac1p+\frac1q=1$).

We focus now our attention on some examples of DH Banach lattices with non-DH duals. The existence of these examples shows that the $\mathfrak{P}$ property cannot be removed from Theorem \ref{duality lemma P}.

\begin{theorem}\label{OrlicznoDH}
Let  $1<p<\infty$ and  $\varphi$ an Orlicz function such that  $\varphi(t)\simeq t^{p}$ on $[0,1]$ and $\varphi(t)\simeq t^{p}\log(1+t)$ on $[1,\infty)$. Then the Orlicz space $L_{\varphi}(0,\infty)$ is a reflexive $p$-DH Banach lattice whose dual is not  DH.
\end{theorem}

The proof of the fact that  $L_{\varphi}(0,\infty)^*$ is not DH is based on a representation of functions in the set $C_{\varphi}(0,\infty)$ given in \cite[p. 242]{Nielsen} and Theorem \ref{DH-Orlicz}. In particular, one can see that this dual Orlicz space contains sublattices isomorphic to the Orlicz sequence space $\ell_{\psi_\alpha}$, for $\psi_\alpha(t)=t^q\abs{\log t}^{\alpha}$, where $\frac1p+\frac1q=1$ and every $\alpha\in(0, \min\{1,q-1\})$.

This example can be used  to construct another  one within the category of \emph{atomic} reflexive $p$-DH Banach lattices, more precisely, a weighted Orlicz sequence spaces.

Recall that given a sequence of positive numbers $w=(w_{n})$  and  an Orlicz function $\varphi$, the \emph{weighted Orlicz sequence space} $\ell_{\varphi}(w)$  is the space of all sequences  $(x_{n})$ such that
\begin{math}
  \sum_{n=1}^{\infty}  \varphi(\frac{\abs{x_n}}{s}) w_n  <  \infty
\end{math}
for some  $s>0$, endowed with the  Luxemburg norm. Notice that the unit vectors form an unconditional basis of $\ell_{\varphi}(w)$ when $\varphi$ satisfies the  $\Delta_{2}$-condition.

\begin{theorem}\label{discrete-dual-non-DH}
Let $w=(w_n)$ be a sequence of positive numbers such that there is  a
subsequence  $(w_{n_k})$ with $w_{n_k}\rightarrow 0$  and  $\sum_{k=1}^{\infty}w_{n_k}=\infty$.
If  $\varphi$ is an Orlicz function as in the previous theorem then the weighted Orlicz sequence space  $\ell_{\varphi}(w)$ is  $p$-DH but its dual is not DH.
\end{theorem}

The proof is based on the space constructed in Theorem \ref{OrlicznoDH} together with an identification theorem for weighted Orlicz sequence spaces \cite{Fuentes-H} and a universal property of these spaces due to \cite{Drewnowski}.

It should be noted that this  kind of examples cannot be adapted to  Orlicz spaces over a probability space (see Theorem \ref{DH-Orlicz} and the comments at begining of this Section). But more generally one might wonder whether a reflexive $p$-DH rearrangement invariant function space (\cite[Chapter 2]{LT2}) on the interval $[0,1]$ whose dual is not DH may exist.

\section{Complemented disjoint sequences}\label{DC}

It was mentioned earlier that Banach lattices in which every positive disjoint sequence has some  subsequence whose span is complemented by a positive projection necessarily satisfy the  $\mathfrak{P}$ property. We take now a closer look at this situation. We will say that a sequence $(x_n)$ is said to be complemented in $E$ if there is a projection $P$ on $E$ with $\Range P=[x_n]$.

Notice that given a positive projection $P$ onto the span of a disjoint sequence $(x_n)\subset E$, if $(x_n^*)$ denote the biorthogonal functionals, then the sequence $(P^*x_n^*)$ need not be disjoint in $E^*$:

\begin{example}
Take
$E=\mathbb R^3$ and let
$$
  x_1=\left[
    \begin{smallmatrix}
      1 \\ 0 \\ 0
    \end{smallmatrix}
    \right],
\hspace{5mm}
  x_2=\left[
    \begin{smallmatrix}
      0 \\ 1 \\ 0
    \end{smallmatrix}
    \right],
\hspace{5mm} \textrm{ and } \hspace{5mm}
  P=\left[
    \begin{smallmatrix}
      1 & 0 & 1 \\
      0 & 1 & 1 \\
      0 & 0 & 0
    \end{smallmatrix}
    \right].
$$
Note that $Pe_1=x_1$, $Pe_2=x_2$, and $Pe_3=x_1+x_2$.
It follows from
\begin{math}
  (P^*x_n^*)_i=\langle P^*x_n^*,e_i\rangle=\langle x^*_n,Pe_i\rangle
\end{math}
that
$$
P^*x_1^*=\left[
    \begin{smallmatrix}
      1 \\ 0 \\ 1
    \end{smallmatrix}
    \right]
\hspace{5mm} \textrm{ and } \hspace{5mm}
  P^*x_2^*=\left[
    \begin{smallmatrix}
      0 \\ 1 \\ 1
    \end{smallmatrix}
    \right],
$$
so that $ P^*x_1^*$ and $ P^*x_2^*$ are not disjoint.
\end{example}

Interestingly enough, the following result proved in \cite{FHSTT} shows that  if a disjoint positive sequence spans a complemented subspace, then a positive projection whose adjoint sends the biorthogonal functionals to a disjoint sequence can be found.

\begin{proposition}\label{positive projection}
 Let $E$ be a reflexive Banach lattice, $(f_n)$ a positive disjoint sequence, and $R\in \mathcal{L}(E)$ a projection onto $[f_n]$. Then there exists a positive disjoint sequence $(g_n^*)$ in $E^*$ with $\langle g_n^*,f_m\rangle=\delta_{n,m}$ such that the operator
\begin{math}
  Px=\sum_{n=1}^\infty g_n^*(x)f_n
\end{math}
defines a positive projection onto $[f_n]$ with $\norm{P}\le\norm{R}$.
\end{proposition}

This fact gains relevance in connection with the following problem: \emph{Does every reflexive Banach lattice contain a complemented positive disjoint sequence?}

We don't know the answer to this question. However, the following result, which is derived from Proposition \ref{positive projection}, provides a useful reformulation.

\begin{corollary}\label{star-property}
  Given a positive disjoint sequence $(e_n)$ in a reflexive Banach lattice $E$, the following are equivalent:
   \begin{enumerate}
   \item The subspace $[e_n]$ is complemented in $E$.
   \item There exists a disjoint positive sequence $(e_n^*)$ in $E^*$ with $\langle e_n^*,e_m\rangle=\delta_{mn}$ such that $\sum_{n=1}^\infty e_n^*(x)e_n$ converges for each $x\in E$.
   \end{enumerate}
\end{corollary}

Note that if $\sum_{n=1}^\infty e_n^*(x)e_n$ converges for every $x\in E$, then the map $P\colon x\mapsto \sum_{n=1}^\infty e_n^*(x)e_n$ defines a positive projection from $E$ onto $[e_n]$. In particular, the above result yields that a reflexive Banach lattice $E$ contains a complemented positive disjoint sequence if and only if $E^*$ does.

The question about the existence of complemented disjoint sequences has a positive answer for most examples of Banach lattices considered in the literature. For instance, if a Banach lattice is atomic (or has an infinite atomic part), that means that $E$ has an unconditional basis inducing the order, and trivially this provides a positive disjoint complemented sequence.

On the other hand, it is well known that in a non-atomic order continuous Banach lattice $E$, every unconditional basic sequence $(u_n)$ spanning a complemented subspace is equivalent to a disjoint sequence $(f_n)$ spanning also a complemented subspace provided that $[u_n]$ is lattice anti-euclidean (that is, $[u_n]$ does not contain uniformly complemented lattice copies of $\ell_2^n$ for every $n$, see \cite[Theorem  3.4]{Casazza-Kalton}).

Another family of spaces which always contain complemented disjoint sequences is that of rearrangement invariant spaces. Using the averaging projection, every sequence of normalized positive characteristic functions over a family of disjoint sets is complemented in any r.i. space (cf. \cite[Theorem 2.a.4]{LT2}).

For DH Banach lattices, the existence of complemented disjoint sequences turns out to be equivalent to the  $\mathfrak{P}$ property studied above. This was proved in \cite[Theorem 4.4]{FHSTT}:

\begin{theorem}\label{star-P}
Let $E$ be a DH Banach lattice. $E$ has the  $\mathfrak{P}$ property if and only if $E$ contains a complemented positive disjoint sequence.
\end{theorem}

In fact, in most instances of DH Banach lattices such as  $L_p$ spaces, Lorentz spaces and some Orlicz spaces, every disjoint sequence has a complemented subsequence. This motivates the following.

\begin{definition}
  A Banach lattice $E$ is called \term{disjointly complemented} (DC) if every disjoint sequence $(x_n)$ has a subsequence whose span is complemented in $E$.
\end{definition}

The study of the relation between DC and DH Banach lattices appears now natural. More specifically, we are interested in deciding whether DH Banach lattices must be  DC.

Let us consider first the non-reflexive case. Recall that if $E$ is non-reflexive, then $E$ either contains a lattice copy of $c_0$ or of $\ell_1$ (cf. \cite[Theorem 2.4.15]{MN}). Therefore, if $E$ is DH and non-reflexive, it follows that it is either $1$-DH or $\infty$-DH.

From Sobczyk's  (\cite[Theorem 2.5.9]{albiac-kalton}), it easily follows that if $E$ is a separable Banach lattice which is $\infty$-DH then it is DC. For the 1-DH case, we will use the following fact (\cite[Lemma~2.3.11]{MN}): If a positive disjoint sequence in a Banach lattice is equivalent to the unit vector basis of $\ell_1$ then its closed span is complemented. It follows  that if $E$ is $1$-DH, then every positive disjoint sequence has a complemented subsequence. Splitting a sequence into its positive and negative parts it can be seen that 1-DH Banach lattices are in fact DC.

Hence, \emph{if $E$ is a separable non-reflexive Banach lattice which is DH, then $E$ is DC}. Clearly, the separability of $E$ is essential here: $\ell_\infty$ is non-reflexive and DH, however it is not DC. In fact, every normalized disjoint sequence is equivalent to the unit vector basis of $c_0$ and by Phillips-Sobczyk's theorem (cf. \cite[Theorem 2.5.5]{albiac-kalton}, \cite[Theorem 2.a.7]{LT1}), the space $\ell_\infty$ does not contain any complemented subspace isomorphic to $c_0$.

Let us consider now the case of reflexive Banach lattices. Does DH imply DC in this context? Recall that in Theorem~\ref{duality lemma P} it was proved that a reflexive Banach lattice $E$ with the  $\mathfrak{P}$ property is DH provided so is $E^*$. The following theorem gives a partial answer to this question. It sumarizes the work done in \cite{FHSTT}.

\begin{theorem}\label{DHandDC}
Let $E$ be a reflexive Banach lattice which  contains a complemented positive disjoint sequence. If  $E$ is DH, then the following are equivalent:
	\begin{enumerate}
	\item[(a)] $E^*$ is DH,
	\item[(b)] $E^*$ has the $\mathfrak{P}$ property,
	\item[(c)] $E^* $ is DC.
	\end{enumerate}

Furthermore, if $E$ and $E^*$ are DH,  then $E$ is DC.

\end{theorem}
Using ideas from \cite{CJZ} it can also be shown that if $E$ is $p$-DH and $p$-convex Banach lattice for some $1\le p<\infty$, then $E$ is DC.

\section{Uniformly DH Banach lattices}\label{sec_uniformlyDH}

Until now  no attention has been given to the equivalence constants involved in the definition of a DH Banach lattice. The purpose of this section is to illustrate the role played by these.

In general an $\ell_p$-sum of $p$-DH spaces need not be DH. Indeed, given $n\in\mathbb{N}$, let $X_n$ denote the completion of the space of all eventually zero sequences $c_{00}$ with respect to the norm
$$
\bignorm{(a_k)}_{X_n}=\sup\Bigl\{\sum_{i=1}^n\abs{a_{k_i}}+\Bigl(\sum_{i>n}\abs{a_{k_i}}^p\Bigr)^{\frac1p}: \,k_1<k_2<\ldots<k_i<\ldots\Bigr\}.
$$

It is easy to see that $\norm{\cdot}_{X_n}$ is equivalent to the $\ell_p$ norm. In fact, we have
\begin{displaymath}
 \bignorm{(a_k)}_{\ell_p}\leq \bignorm{(a_k)}_{X_n}  \leq (n^{\frac1q}+1)\bignorm{(a_k)}_{\ell_p}.
\end{displaymath}

In \cite[Example 6.4]{FHSTT} it was proved that the space $\Bigl(\bigoplus_{n=1}^\infty X_n\Bigr)_{\ell_p}$ endowed with the $\ell_p$-sum of the corresponding norms $\|\cdot\|_{X_n}$ is not DH.

Note that in the above example, the equivalence constant of disjoint sequences in different $X_n$-summands grows without bound. After this  example it seems only natural to introduce the following:

\begin{definition}
A Banach lattice $E$ is \term{uniformly disjointly homogeneous} if there is a constant $C>0$ such that every two disjoint normalized
sequences $(x_n)$ and $(y_n)$ in $E$, have subsequences such that $(x_{n_k})\overset{C}\sim(y_{n_k})$.
\end{definition}

Clearly every uniformly DH Banach lattice $E$ is DH. The converse is nevertheless not true. In fact,
 stemming from deep results of W. B. Johnson and E. Odell in \cite{JO}, and H. Knaust and E. Odell in \cite{KO} the following result is given in \cite{FHSTT}

\begin{theorem}
For every $1<p<\infty$, there exists a super-reflexive atomic Banach lattice $E_p$ which is $p$-DH but not uniformly DH.
\end{theorem}

As a by-product,  another example of a reflexive DH Banach lattice whose dual is not DH is obtained.

In Theorem~\ref{discrete-dual-non-DH} we have constructed examples of reflexive atomic Banach lattices (with the order induced by a $1$-unconditional basis), which are DH, but whose dual spaces are not. The case of atomic Banach lattices with the order induced by a subsymmetric basis deserves some attention. Recall that a basis $(x_n)$ is called \emph{subsymmetric} if it is unconditional and every subsequence $(x_{n_i})$ is equivalent to $(x_n)$ (cf. \cite[Chapter 3]{LT1}).

Also, recall that a normalized basis $(e_n)$ in a Banach space $X$ is said to be a \term{Rosenthal basis} if every normalized block-sequence of $(e_n)$ contains a subsequence equivalent to $(e_n)$. It is an open question whether such a basis is necessarily equivalent to the unit basis of $\ell_p$ or $c_0$, see~\cite{FPR} for further details and partial results in this direction. In particular, it was observed in \cite[p.~397]{FPR} that a Rosenthal basis $(x_n)$, always satisfies that every subsequence $(x_{n_i})$ is equivalent to $(x_n)$.

\begin{proposition}\label{DH-Ros}
  Let $E$ be a reflexive atomic Banach lattice with the order induced by a subsymmetric basis $(e_n)$. Then $E$ is DH if and only if $(e_n)$ is a Rosenthal basis.
\end{proposition}

Let $X$ be a Banach space with a Rosenthal basis $(e_n)$. It was proved in~\cite[Theorem 1, Proposition 7]{FPR} that $(e_n)$ is equivalent to the unit basis of $\ell_p$ or $c_0$ if $(e_n)$ is ``uniformly'' Rosenthal or if $E^*$ also has a Rosenthal basis. In view of Proposition~\ref{DH-Ros}, we can now restate these statements in terms of disjoint homogeneity as follows.

\begin{proposition}\label{DH-Ros-app}
  Let $E$ be a reflexive atomic Banach lattice with the order induced by a subsymmetric normalized basis $(e_n)$. Then $(e_n)$ is  equivalent to the unit basis of $\ell_p$ for some $1<p<\infty$ if any of the following conditions is satisfied:
  \begin{enumerate}
  \item $E$ is uniformly DH, or
  \item\label{DH-Ros-app-dual} $E$ and $E^*$ are both DH.
  \end{enumerate}
\end{proposition}

In particular, if $(e_n)$ is symmetric, then Proposition~\ref{DH-Ros-app}\eqref{DH-Ros-app-dual} also follows from \cite[Theorem 3.a.10]{LT1} due to Z. Altshuler. Indeed, if $E$ is DH and $(v_n)$  is a sequence generated by one vector (which is automatically symmetric), then $(v_n)$ and $(e_n)$ have equivalent subsequences, hence are themselves equivalent. Now apply the same argument to $(e_n^*)$ in $E^*$.

We do not know whether every atomic reflexive Banach lattice with the order induced by a subsymmetric basis which is DH must be isomorphic to $\ell_p$ for some $1<p<\infty$. In this direction, if we consider the symmetric version of Tsirelson space (see \cite[Chapter X, B]{Casazza-Shura}), which does not contain $\ell_p$ subspaces, then it is not hard to see that this space fails being DH. However, let us suppose that $E$ is a reflexive Banach lattice with the $\mathfrak{P}$ property containing a disjoint subsymmetric sequence, if $E^*$  is $DH$, then $E$ must be $p$-$DH$ for some $1<p<\infty$ (see \cite[Corollary 6.10]{FHSTT}).

\section{Compact powers of strictly singular operators}\label{section compact squares}

The purpose of this section is to show how DH Banach lattices can be applied to the theory of strictly singular operators. In particular, we are interested in the extension of a result by V. Milman \cite{Milman}, which asserts that every strictly singular endomorphism on $L_p$ has compact square. This kind of results have also been studied in \cite{ADST} in the context of Banach spaces.

 Recall that an operator between Banach spaces is \emph{strictly singular} if it is not invertible on any infinite dimensional subspace. This is an important class of operators which was first introduced in connection with the perturbation of Fredholm operators \cite{Kato}, and has later proved relevant  in the modern theory of Banach spaces (see \cite{Argyros-Haydon}).

Given a Banach space $X$, we will denote by $\mathcal{K}(X)$ (respectively $\mathcal{S}(X)$) the space of all compact  (resp. strictly singular) endomorphisms on $X$.

A close notion to strict singularity was introduced in the setting of Banach lattices (\cite{Hernandez-Salinas}): given a Banach lattice $E$ and a Banach space $X$, an operator $T:E\rightarrow X$ is \emph{disjointly strictly singular} (DSS) if for any sequence of pairwise disjoint elements $(x_n)$ in $E$, the restriction of $T$ to the span $[x_n]$ is not invertible.  Recall also that an operator $T:E\rightarrow X$ is \textit{AM-compact} whenever $T([-x,x])$ is a relatively compact set in $X$ for every $x\in E_+$ (recall that the order interval $[-x,x]$ is the set $\{y\in E:|y|\leq x\}$).

In \cite{FTT2} several results about compactness of operators belonging to the singular classes  given above were proved in the context of regular operators. Recall that an operator between Banach lattices is positive when it maps positive elements to positive elements, and a regular operator is a difference of two positive ones.

\begin{theorem}\label{compact}
  Suppose that $E$ is a DH Banach lattice with order continuous norm and a weak unit. Suppose that $S$ and $T$ are
  two regular operators on $E$ such that $S$ is disjointly strictly
  singular and $T$ is AM-compact.
  \begin{enumerate}
  \item\label{dual-oc} If $E^*$ is order continuous then $ST$ is compact.
  \item\label{dual-nooc} If $E^*$ is not order continuous then $TS$ is
    compact.
  \end{enumerate}
  In particular, if $R$ is disjointly strictly singular and
  regular, then $STR$ is compact.
\end{theorem}

Observe that Theorem~\ref{compact}\eqref{dual-nooc} remains valid in the case that $S$ is not regular. Also, it remains valid if, instead of being disjointly strictly singular, $S$ is only assumed to be weakly compact. In particular, the above result yields that if $E$ is DH and $T:E\rightarrow E$ is regular, disjointly strictly singular, and AM-compact, then $T^2$ is compact.

A Banach lattice $E$ has finite cotype (or equivalently finite concavity) if and only if $E$ does not contain copies of $\ell_\infty^n$ uniformly (cf. \cite{LT2}). Moreover, every  Banach lattice $E$ with finite concavity satisfies the \textit{subsequence splitting property} (\cite{WeisSSP}). This means that every bounded sequence $(x_n)$ in $E$ has a subsequence that can be written as $x_{n_k}=g_k+h_k$, with $|g_k|\wedge|h_k|=0$, the sequence $(g_k)$ being equi-integrable and $(h_k)$ disjoint. Recall that a bounded sequence $(g_n)$ in a Banach lattice of measurable functions over a measure space $(\Omega,\Sigma,\mu)$ is equi-integrable if $\sup_n\|g_n\chi_A\|\rightarrow 0$ as $\mu(A)\rightarrow 0$. Note that every Banach lattice with finite cotype is order continuous.

\begin{theorem}\label{dh-ssp-dual}
Let $E$ be a DH Banach lattice with the subsequence splitting property, such that $E^*$ is order continuous. If $T:E\rightarrow E$ is a regular operator which is disjointly strictly singular and AM-compact, then $T$ is compact.
\end{theorem}

In \cite{FHST} several results in similar spirit were given without the restriction of regularity. An important technique that was exploited in these arguments is the well known Kadec-Pe\l czy\'nski's dichotomy (see \cite{FJT}, \cite{LT2}): given a normalized sequence $(x_n)$ in an order continuous Banach lattice $E$
\begin{enumerate}
    \item either $(\|x_n\|_{L_1})$ is bounded away from zero,
    \item or there exist a subsequence $(x_{n_k})$ and a disjoint sequence $(z_k)$ in $E$ such that $\|z_k-x_{n_k}\|\longrightarrow 0$ as $k\rightarrow\infty$.
\end{enumerate}

An operator $T:E\rightarrow X$ is called M\textit{-weakly compact} if it maps disjoint sequences in $B_E$ to sequences converging to zero. Notice that an operator is compact if and only if it is AM-compact and M-weakly compact (\cite[Proposition 3.7.4]{MN}). There is a notion dual to M-weak compactness, namely, an operator $T:X\rightarrow E$ is L\textit{-weakly compact} if every disjoint sequence in the solid hull of $T(B_X)$ tends to zero in norm. The following fact was given in \cite{Tradacete} for endomorphism on $L_p$ spaces.

\begin{proposition}\label{positive DSS Mw Lw}
Let $E$ be a reflexive DH Banach lattice and $T:E\rightarrow E$ be a positive operator. The following are equivalent:
\begin{enumerate}
\item $T$ is disjointly strictly singular.
\item $T$ is M-weakly compact.
\item $T$ is L-weakly compact.
\end{enumerate}
\end{proposition}

\begin{proof}
An M-weakly compact operator is clearly disjointly strictly singular. For the converse implication, suppose that $T$ is not M-weakly compact. Thus, there is a disjoint normalized sequence $(x_n)$ in $E$ such that $\|Tx_n\|_E\geq\alpha>0$.

Observe that $(|x_n|)$ is also a disjoint normalized sequence, and since $E$ is reflexive it must be weakly null. Hence, so is $(T|x_n|)$, and since $T$ is positive it follows that $\|T|x_n|\|_{L_1}\rightarrow0$. Note that
$$
\|T|x_n|\|_E\geq\|Tx_n\|_E\geq\alpha>0,
$$
so by Kadec-Pe\l czy\'nski's dichotomy, $(T|x_n|)$ has a subsequence equivalent to a disjoint sequence in $E$. Since $E$ is DH, $(|x_n|)$ and $(T|x_n|)$ have an equivalent subsequence and $T$ is not DSS. This proves the equivalence of the first two statements. The equivalence with the third one follows from \cite[Theorem 3.6.17]{MN}.
\end{proof}

Recall that an operator between Banach spaces  $T:X\rightarrow Y$ is \textit{Dunford-Pettis} if it maps weakly null sequences to sequences converging to zero. The following result from \cite{FHST} can be seen as an extension of the classical result stating that weakly compact operators on $L_1$ are Dunford-Pettis.

\begin{theorem}\label{1DH-DunfordPettis}
Let $E$ be a 1-DH Banach lattice with finite cotype. Every operator $T\in\mathcal{S}(E)$ is Dunford-Pettis.
\end{theorem}

Using that the composition of a weakly compact with a Dunford-Pettis operator is a compact operator, it follows that every strictly singular operator on a 1-DH Banach lattice with finite cotype has compact square.

Before we present the extension of Milman's result (\cite{Milman}) on compactness of the square of strictly singular endomorphisms on $L_p$ spaces we need the following.

\begin{definition}
A Banach lattice $E$ has property $(C)$ if it is order continuous, and there exist $q<\infty$ and a probability space $(\Omega,\Sigma,\mu)$ such that the inclusions $L_q(\mu)\hookrightarrow E\hookrightarrow L_1(\mu)$ hold.
\end{definition}

Note that condition $(C)$ is a very mild assumption. Indeed, every Banach lattice with a weak order unit (for instance separable) and finite cotype satisfies property $(C)$ (see \cite[p. 14]{JMST}). Moreover, every order continuous rearrangement invariant function space on $[0,1]$ with upper Boyd index $q_X<\infty$ also has property $(C)$ (though it may have trivial cotype, \cite[Proposition 2.b.3]{LT2}). Recall that for an r.i. function space $X$ the {\it Boyd indices} are given by
$$
p_X=\lim_{s\rightarrow\infty}\frac{\log s}{\log\|D_s\|}\hspace{2cm}q_X=\lim_{s\rightarrow0^+}\frac{\log s}{\log\|D_s\|},
$$
where $D_s:X\rightarrow X$ is the dilation operator given by $(D_sf)(t)=f(t/s)$ for $t\leq\min(1,s)$ and zero otherwise (see \cite[Section 2.b]{LT2}).

In \cite{FHST}, the following was proved.

\begin{proposition}\label{DH square AM-compact}
Let $E$ be a DH Banach lattice with property $(C)$. If $T\in\mathcal{S}(E)$ then $T^2$ is AM-compact.
\end{proposition}

This is a first step in the proof of the next theorem also from \cite{FHST}.

\begin{theorem}\label{sq compact}
Let $E$ be a DH Banach lattice with finite cotype and an unconditional basis. Every operator $T\in \mathcal{S}(E)$ satisfies that the square $T^2$ is compact.
\end{theorem}

The existence of an unconditional basis is a technical condition in the previous proof and indeed not a restriction for many spaces. We conjecture that the result is still true without it. In fact, there are some situations in which there is no need to impose it:

\begin{theorem}\label{pDH p>2}
If $E$ is a $p$-DH Banach lattice ($2\leq p\leq \infty$) with property $(C)$, then every operator $T\in\mathcal{S}(E)$ has compact square.
\end{theorem}

A classical result of J. Calkin \cite{Calkin} states that the only non-trivial closed ideal of operators in Hilbert space is the ideal of compact operators. In particular, as pointed out by T. Kato \cite{Kato}, on Hilbert spaces the ideals of strictly singular and compact operators coincide. In fact, this is a particular case of a more general result involving $2$-DH Banach lattices (\cite[Theorem 2.12]{FHST}):

\begin{theorem}\label{2DH}
If $E$ is a $2$-DH Banach lattice with property $(C)$, then $\mathcal{S}(E)=\mathcal{K}(E)$.
\end{theorem}

To finish this section, the case of atomic Banach lattices deserves its proper space. Note that in the atomic case  the class of DH Banach lattices $E$  with a basis of disjoint vectors is a rather  small class,  since ``most'' basic sequences in $E$ are equivalent to disjoint sequences. As mentioned earlier the examples include $\ell_p$ spaces and $c_0$, and also Schreier, Baernstein, Tsirelson spaces and their generalizations (cf. \cite{Casazza-Shura}), as well as $\ell_p$-sums of finite dimensional Banach lattices $\ell_p(X_n)$. Also mentioned earlier was that Lorentz and Orlicz sequence spaces (distinct from spaces $\ell_{p}$) are not disjointly homogeneous.

In the atomic setting,  Theorem \ref{sq compact} is improved, similarly to the case of $\ell_p$ spaces where strictly singular and compact endomorphisms  coincide (cf. \cite[p. 76]{LT1}). This is shown in the following result (\cite{FHST})

\begin{theorem}\label{discrete}
Let $E$ be an atomic Banach lattice with a basis. If $E$ is DH then every operator $T\in\mathcal{S}(E)$ is compact.
\end{theorem}

\section{Applications to operators on rearrangement invariant spaces}\label{section lorentz}

In the class of rearrangement invariant spaces it happens that the behaviour of powers of endomorphisms determines the behaviour of the composition of (different) operators. This is the content of the next result (see \cite{FHST} for details).

\begin{proposition}
Given a rearrangement invariant space $E$ on $[0,1]$ and $n\in\mathbb{N}$ the following statements are equivalent:
\begin{enumerate}
\item If an operator $T\in\mathcal{S}(E)$, then the power $T^n$ is compact.
\item If $T_1,\ldots, T_n$ belong to $\mathcal{S}(E)$, then the composition $T_n\ldots T_1$ is compact.
\end{enumerate}
\end{proposition}

As a consequence of Theorems \ref{1DH-DunfordPettis}, \ref{sq compact} and \ref{2DH} we get the following:

\begin{proposition}
Given $1<p<\infty$ and $1\leq q<\infty$, every operator $T\in\mathcal{S}(\Lambda(W,q)[0,1])$ or $T\in \mathcal{S}(L_{p,q}[0,1])$ has a compact square. Moreover, if $q=2$ then $T$ is already compact, while if $q=1$, then $T$ is Dunford-Pettis.
\end{proposition}

By contrast, if $q\neq 2$, then strictly singular non-compact operators on $L_{p,q}$ can be found. Take for instance a complemented subspace isomorphic to $\ell_q$ and the span of the Rademacher functions which is isomorphic to $\ell_2$. Denote $P_1:L_{p,q}\rightarrow \ell_q$ and $P_2:L_{p,q}\rightarrow \ell_2$  the corresponding projections, $i_{s,t}:\ell_s\rightarrow \ell_t$  the canonical inclusion, and $Q:\ell_q\hookrightarrow L_{p,q}$ and $R:\ell_2\hookrightarrow L_{p,q}$  the corresponding embeddings. When $q<2$   consider $T=Ri_{q,2}P_1\in
\mathcal{S}(L_{p,q})\backslash\mathcal{K}(L_{p,q})$, and when $q>2$ take  $S=Qi_{2,q}P_2\in \mathcal{S}(L_{p,q})\backslash\mathcal{K}(L_{p,q})$.

The behavior of the maximal Lorentz spaces  $L_{p,\infty}$ is quite different. In fact, there exists an operator $T\in \mathcal{S}(L_{p,\infty})$, for $p\neq 2$, whose cube $T^3$ is not compact. The proof is based on a particular  way of embedding  $\ell_p$ as a complented subspace into $L_{p,\infty}$ (notice that for $p<2$  even $L_p$ can be embedded as a complemented subspace of $L_{p,\infty}$, see \cite{KaltonL0}), and the fact that $\ell_{p,\infty}$ embeds as a complemented sublattice into $L_{p,\infty}$ \cite{Leung}. See \cite[Proposition 3.3]{FHST} for details.

Similarly  strictly singular operators on $L_{2,\infty}$ with non-compact squares can be defined. However, we do not know whether there might exist some $n\in\mathbb{N}$ such that $T^n$ is compact whenever $T\in \mathcal{S}(L_{p,\infty})$, or even whether every operator $T\in\mathcal{S}(L_{p,\infty})$ is power-compact.

Observe also that  if $L_{p,\infty}^o$ denotes the order continuous part of $L_{p,\infty}$, then every strictly singular operator on $L_{p,\infty}^o$ has compact square. This follows from Theorem \ref{pDH p>2}, since $L_{p,\infty}^o$ is $\infty$-DH and his upper Boyd index equals $p$. A similar statement also holds for order continuous Marcinkiewicz spaces $M(\varphi)$ with finite upper Boyd index (since they are also $\infty$-DH Banach lattices, cf. \cite{semenov}).

Note  however  that these results do not hold for Lorentz spaces $L_{p,q}(0,\infty)$ (for $p\neq q$) as they contain complemented lattice copies of the non-DH spaces $\ell_{p,q}$.

In the case of Orlicz spaces over a probability measure space $L_\varphi$ is DH if and only if $E_\varphi^\infty\cong\{t^p\}$ (Theorem \ref{DH-Orlicz}). This condition  implies the equality of the indices $s(L_\varphi)=\sigma(L_\varphi)=p,$ or equivalently the equality of the associated Boyd indices $p_{L_\varphi}$ and $q_{L_\varphi}$ as it follows from the identities $s(L_\varphi)=p_{L_\varphi}$ and $\sigma(L_\varphi)=q_{L_\varphi}$  (cf. \cite[p. 139]{LT2}). Thus, from Theorems \ref{1DH-DunfordPettis} and \ref{sq compact} the following result is obtained.

\begin{proposition}
Let $\varphi$ be an Orlicz function such that $E_{\varphi}^\infty\cong\{t^p\}$, for some $1\leq p<\infty$. If an operator $T\in \mathcal{S}(L_\varphi[0,1])$ then the square $T^2$ is compact. Furthermore for $p=2$, the operator $T$ is already compact, while for $p=1$, $T$ is Dunford-Pettis.
\end{proposition}

Many Orlicz functions satisfy the condition $E_\varphi^\infty\cong\{t^p\}$, for example the class of all Orlicz functions of regular variation, i.e.  $$\lim_{t\rightarrow\infty}\frac{t\varphi'(t)}{\varphi(t)}=p.$$ In general, we cannot weaken this condition on $E_\varphi^\infty$, as there exist Orlicz spaces $L_\varphi$ with indices $s(L_\varphi)=\sigma(L_\varphi)=p$, and an operator $T\in\mathcal{S}(L_\varphi)$ whose square $T^2$ is not compact (while for $p=2$, we have a strictly singular non-compact operator, see \cite[Proposition 4.3]{FHST}).

Clearly these Orlicz spaces  are not disjointly homogeneous (this follows from Theorem \ref{sq compact}). More generally, every {\it minimal} Orlicz function space $L_{\varphi}$ (different from $L_{p}$) is not disjointly homogeneous. Indeed,  recall that in general for each $\psi \in C_{\varphi}^{\infty}$ there exists a sequence of normalized disjoint functions in $L_{\varphi}$ equivalent to the symmetric canonical basis of $\ell_{\psi}$ (\cite[Proposition 4]{LT3}). Now, since $\varphi$ is minimal, we have, by \cite[Proposition 1]{Hernandez-Peirats}, that $E_{\varphi,1}^{\infty} = E_{\varphi}^{\infty}=E_{\varphi}=E_{\varphi,1}$ and the set $E_{\varphi,1}^{\infty}$ contains uncountable many mutually
non-equivalent Orlicz functions (see the proof of \cite[Theorem 4.b.9]{LT1}). Hence, using the symmetry, we deduce that in $L_{\varphi}$ there are uncountable many sequences of normalized disjoint functions with no equivalent subsequence.

Notice also that in the class of Orlicz spaces $L_\varphi$ with different indices ($s(L_\varphi)\neq\sigma(L_\varphi)$) there are no DH spaces. This follows from the fact that for each $p\in[s(L_\varphi),\sigma(L_\varphi)]$ we have $t^{p}\in C_{\varphi}^{\infty}$ and there exist sequences of normalized disjoint functions in $L_{\varphi}$ that are equivalent to the canonical basis of $\ell_{p}$ (\cite[Proposition 4]{LT3}).

New examples of DH r.i. spaces and its connection with interpolation theory can be found also in the recent paper \cite{Astashkin}.

\section{The Kato property in rearrangement invariant spaces}\label{section kato}

 As mentioned above the ideals of strictly singular and compact operators coincide on  Hilbert spaces as well as on 2-DH  function spaces (under  some mild  assumptions).

In this section we consider the natural converse question:
Assume that  $E$ is  an r.i. function space  such that every strictly singular operator \,$T\in\LL(E)$ is compact. Must $E$ be 2-DH?

This question has been addressed in \cite{HST}. First, let us introduce the following

\begin{definition}
A Banach space $X$ has the \textsl{Kato property} whenever $\Ss(X)=\K(X).$
\end{definition}

Examples of function spaces with the Kato property clearly include Hilbert spaces, Lorentz spaces of the form \,$L_{p,2}[0,1]$\, and \,$\Lambda(W,2)[0,1]$\, as well as certain  Orlicz spaces \,$L_\varphi[0,1]$ \, like  $\varphi(t)=t^2\log^\alpha(1+t)$ for arbitrary $\alpha$. The  Kato property is also enjoyed by  the sequence spaces  $\ell_p$ ($1\leq p<\infty $)\, $c_0$, the Tsirelson space (and some of its modifications) and also some new  spaces such as the space $\mathcal{X}_{AH}$\, (constructed in \cite{Argyros-Haydon} as a solution to the scalar-plus-compact problem).
Notice that the Kato property is an isomorphic property. Moreover, we have the following:

\begin{proposition}
Let $X$ be a Banach space with the Kato property. Suppose that for some subspace $Y\subset X$, there is $Z\subset X$ such that $Y\simeq X/Z$, then $Y$ also has the Kato property. In particular, every complemented subspace of a space with the Kato property also has the Kato property.
\end{proposition}

A  weaker version of the 2-DH  property is the following :

\begin{definition}
An  r.i. space on $[0,1]$ is  \emph{restricted 2-DH} if for every sequence of disjoint sets $(A_n)_{n=1}^\infty$ in $[0,1]$ there is a subsequence such that  $(\frac{1}{\|\chi_{A_{n_k}}\|}\chi_{A_{n_k}})_{k=1}^\infty$ is equivalent to the unit vector basis of $\ell_2$.
\end{definition}

\begin{proposition}\label{weak2DH-HaarDH}
Let $E$ be an r.i. space on $[0,1]$. The space $E$ is restricted 2-DH if and only if every subsequence of disjoint elements of the normalized Haar basis in $E$ has a further subsequence equivalent to the unit vector basis of $\ell_2$.
\end{proposition}
\vspace{1mm}
This implies the following:
{\it If  $E$ is an r.i. space on $[0,1]$ which is isomorphic  to a 2-DH r.i. space $F$, then  $E$\, is restricted 2-DH}. Indeed, this follows from \cite[Theorem 6.1]{JMST}, since either $E=F$ up to equivalence of norms, or $E=L_2[0,1]$, or the Haar basis in $E$ is equivalent to a sequence of disjoint elements in $F$ and, in this case,  the result is a consequence of Proposition \ref{weak2DH-HaarDH}.

As  mentioned in Section \ref{stab-dual}, the  2-DH property is not stable in general by duality. However, restricted  2-DH\  r.i. spaces on $[0,1]$ are stable under duality.
We do not know if in general an r.i. space $E$ on $[0,1]$ which is restricted 2-DH, must be 2-DH. In the  class of Orlicz spaces $L_{\varphi}[0,1]$ \, this is  the case:

\begin{proposition}\label{Orlicz2DH}
For an Orlicz space \,$L_{\varphi}[0,1]$, the following are equivalent:
\begin{enumerate}
\item $L_{\varphi}[0,1]$ is 2-DH.
\item $L_{\varphi}[0,1]$ \,  is restricted 2-DH.
\item Every function in $E_\varphi^\infty$ is equivalent to the function $t^2$ at $0$.
\end{enumerate}
\end{proposition}

For non-reflexive r.i. spaces we can use Lozanovski's theorem and the factorization through  $\ell_{1}$ and $c_{0}$ to obtain the following

\begin{proposition}\label{non-reflexive Kato}
If  $E$ is a non-reflexive r.i. space on $[0,1]$, then $E$  fails to have the Kato property.
\end{proposition}

In the class of Lorentz function spaces we have that \emph{$L_{p,q}[0,1]$ and $\Lambda(W,p)[0,1]$ have the Kato property if and only if they are 2-DH}.

On the other hand, for Orlicz spaces the study of the Kato property is more involved. Clearly every  2-DH  Orlicz space on \,$[0,1]$\, has the Kato property. Remarkably              {\it if  \,$L_\varphi [0,1]$\,  is a reflexive 2-convex (or  2-concave) Orlicz space with  the Kato property then  it is 2-DH}. The proof uses the fact that the  associated  Orlicz sequence space $\ell_{\psi}$ is $2$-convex (or 2-concave), so that the inclusion $\ell_{2} \hookrightarrow \ell_{\psi}$ \,(or $\ell_{\psi} \hookrightarrow \ell_{2}$) is  strictly singular. Finally, by composing with canonic projections we get a non-compact strictly singular operator on  $L_\varphi [0,1]$.

\vspace{1mm}

    The  condition $(iii)$ in Proposition \ref{Orlicz2DH} which characterizes 2-DH  Orlicz spaces can be rewritten as the formula
$$
\sup_{0<t<\infty}\limsup_{u\rightarrow\infty} \frac{\varphi(tu)}{t^2\varphi(u)}<\infty.
$$
Strengthening slightly this condition we get further necessary conditions for the Kato property in  Orlicz spaces.

\begin{theorem}
Let $L_\varphi[0,1]$ be a reflexive Orlicz space. If
$$
\lim\limits_{t\to0}\lim\limits_{u\to\infty} \frac{\varphi(tu)}{t^2\varphi(u)}\in\{0,\infty\},
$$
then $L_\varphi[0,1]$
 fails to have the Kato property.
\end{theorem}

The proof relies on factorization  results. Thus, in the case that  $\varphi$ satisfies
$$
\lim\limits_{t\to0}\lim\limits_{u\to\infty} \frac{\varphi(tu)}{t^2\varphi(u)}=\infty.
$$
it can be shown, using N. Kalton's characterization for strictly singular inclusions between Orlicz sequence spaces  (cf. \cite{LT1}), that  there exists a sequence of disjoint measurable sets $(A_k)$ in $[0,1]$, and an
Orlicz function $\psi$  such that the sequence $(\chi_{A_k}/\|\chi_{A_k}\|)$\,  is equivalent to the unit vector basis of the sequence space \,$\ell_\psi$; in addition, the inclusion $\ell_\psi\subset
\ell_2$ holds and it  is  strictly singular.

\vspace{2mm}

We do not know whether every Orlicz space $L_\varphi[0,1]$ with the Kato property must be 2-DH. Notice that for infinite measures the answer is negative as shown with the following

\begin{example}
Consider
$$
\varphi(t)=
\left\{
\begin{array}{ccc}
  \frac{1}{\log 2}t^{2} &   &  t \in [0,1]  \\
  &   &    \\
 \frac{t^{2}}{log \, ( 1+t)} &   &    t\in [1, \infty).
\end{array}
\right.
$$
Then the reflexive space $L_\varphi(0,\infty)$  has the Kato property but is not  2-DH.
\end{example}

Indeed, note that $L_{\varphi}(0,\infty)$  is isomorphic to $ L_{\varphi}[0,1]$  (\cite{JMST}).  But  $ L_{\varphi}[0,1]$ is  2-DH and thus it has the Kato property.  On the other hand, it can be shown that the function \, $t^{2}log(|log t|)$\, belongs to  the set \, $C_{\varphi}(0,\infty)$. Hence,  using Theorem \ref{DH-Orlicz} we conclude  that the space  $L_{\varphi}(0,\infty)$ is not 2-DH.

\section{Some open questions }

In this final section we collect several questions that arised along the preceding sections.

\begin{question}
Does every reflexive Banach lattice contain a disjoint sequence whose span is complemented?
\end{question}

\begin{question}
Is every separable DH Banach lattice DC?
\end{question}

\begin{question}
Is there a reflexive $p$-DH r.i. space on $[0,1]$ whose dual is not DH?
\end{question}

\begin{question}
Is every DH atomic Banach lattice with a symmetric basis isomorphic to $l_p,  (1\le p<\infty)$ or $c_0$?
\end{question}

\begin{question}
If $X$ is an r.i. space on $[0,1]$ with the Kato property, must $X$ be 2-DH? In particular, is every Orlicz space $L_{\varphi}[0,1]$ with the Kato property necessarily 2-DH?
\end{question}

\section*{Acknowledgments}
\thanks{We want to thank E. M. Semenov,  E. Spinu and V. G. Troitsky for all the stimulating discussions that motivated many of the results included in this survey.}

\end{document}